\documentclass[12pt]{amsart}
\usepackage{amssymb}
\usepackage{latexsym}
\usepackage{amsfonts}
\usepackage{amsmath}
\usepackage{amssymb}

  \newtheorem{theorem}{Theorem}[section]

\newtheorem{lemma}[theorem]{Lemma}

\theoremstyle{definition}

\def\W{{\Omega}}

\def\la{{\langle}}
\def\ra{{\rangle}}
\def\char{{\rm char \,}}

\def\a{\alpha}
\def\b{\beta}
\def\e{\varepsilon}

\renewcommand{\leq}{\leqslant}
\renewcommand{\geq}{\geqslant}

\newcommand{\lcs}[2]{#1^{#2}}

\newcommand{\F}{\mathbb F}

\newcommand{\gr}[1]{{\sf Gr}(#1)}

\newcommand{\aug}[1]{\W(#1)}
\newcommand{\augp}[2]{\Omega^{#2}(#1)}
\newcommand{\uea}[1]{U(#1)}

\begin{document}

\title[Universal enveloping algebras of nilpotent Lie algebras]{The isomorphism problem for universal enveloping algebras of
nilpotent Lie algebras}
\author{Csaba Schneider and Hamid Usefi}
\address[Schneider]{Centro de \'Algebra da Universidade de Lisboa\\
Av. Prof. Gama Pinto, 2\\ 1649-003 Lisboa\\ Portugal\hfill\break
Email: csaba.schneider@gmail.com\hfill\break
WWW: www.sztaki.hu/$\sim$schneider}
\address[Schneider]{Informatics Research Laboratory\\
Computer and Automation Research Institute\\
1518 Budapest Pf.\ 63\\
Hungary}
\address[Usefi]{Department of Mathematics\\
University of British Columbia\\
1984 Mathematics Road\\
Vancouver, BC, Canada, V6T 1Z2\hfill\break
Email: usefi@math.ubc.ca\hfill\break
WWW: www.math.ubc.ca/$\sim$usefi}

\thanks{The first author was supported by the FCT project
PTDC/MAT/101993/2008 and by the Hungarian
Scientific Research Fund (OTKA) grant~72845. The second author is supported by an NSERC PDF}

\begin{abstract}
In this paper we study the isomorphism problem for the
universal enveloping algebras of nilpotent Lie algebras. We prove that
if the characteristic of the underlying field is not~2 or~3, then
the isomorphism type of a nilpotent Lie algebra of dimension at most~6 is
determined by the isomorphism type of its universal enveloping algebra. Examples show that the
restriction on the characteristic is necessary.
\end{abstract}


\maketitle

\section{Introduction}\label{intro}
In this paper we examine the isomorphism problem for universal
enveloping algebras of Lie algebras. It is known that two non-isomorphic Lie
algebras may have isomorphic universal enveloping algebras; see for instance
\cite[Example~A]{ru}. However, all such known examples require that the
 characteristic
of the underlying field is a prime.  In this paper we focus on
nilpotent Lie algebras and prove the following main result.

\begin{theorem}\label{main}
Suppose that $\F$ is a field such that $\char\,\F\not\in\{2,3\}$. Then
the isomorphism type of a nilpotent Lie algebra of dimension at most~$6$
over~$\F$ is determined by the isomorphism type of its universal
enveloping algebra.
\end{theorem}

Theorem~\ref{main} is a consequence of Theorems~\ref{5main} and~\ref{6main}.
As shown by examples in Sections~\ref{dim5} and~\ref{dim6},
the requirement about the characteristic of the underlying field is necessary.
In addition to proving Theorem~\ref{main}, we classify the possible
isomorphisms between the universal enveloping algebras of nilpotent Lie
algebras of dimension at most~5 over an arbitrary field, and those of
dimension~6 over fields of characteristic different from~2 (see Theorems~\ref{5main} and~\ref{6main}).

Little progress has ever been made on the isomorphism problem for
universal enveloping algebras. For a Lie algebra $L$, let $\uea L$ denote
its universal enveloping algebra (see Section~\ref{prelim} for the definitions).
Several invariants of $L$
are known to be determined by $\uea L$. For instance,
if $L$ is  finite-dimensional, then the (linear) dimension of $L$
coincides with the Gelfand-Kirillov dimension of $\uea L$ (see~\cite{KL}).
More recently~Riley and Usefi~\cite{ru} proved that the nilpotence of $L$ is
determined by $U(L)$ and, for a nilpotent $L$, the nilpotency
class of $L$ can be computed using $U(L)$.
Moreover the isomorphism type of $U(L)$ determines the isomorphism
type of the graded algebra $\gr L$ associated with the lower central series of
$L$ (see Section~\ref{prelim} for the definitions).
Malcolmson~\cite{mal}
showed that if $L$ is a 3-dimensional simple Lie algebra over
a field of characteristic not~2, then
$L$ is determined by $U(L)$ up to isomorphism.
Later Chun, Kajiwara, and Lee~\cite{ckl} generalized Malcolmson's result
to the class of all Lie algebras with dimension $3$
over fields of characteristic not~2.

By proving Theorem~\ref{main},
we verify  that the isomorphism
problem for universal enveloping algebras has a positive solution in the class
of nilpotent Lie algebras with dimension at most~6 over fields of
characteristic
different from 2~and~3.  The proof of this result relies on
the classification of nilpotent Lie
algebras with dimension at most 6. The classification of such Lie algebras
of dimension at most~5 has been known for a long time over
an arbitrary field. In dimension~6, several classifications have been published,
but they were often incorrect, and they usually only treated fields of
characteristic~0. Recently De Graaf~\cite{degraaf} published a
classification of 6-dimensional nilpotent Lie algebras over an arbitrary field
of characteristic not~2. As the classification by De Graaf has been obtained
making heavy use of computer calculations, and was checked
by computer for small
fields~\cite{sch}, we consider this classification as the most reliable in
the literature. The reason we do not treat 6-dimensional nilpotent Lie
algebras  over fields of characteristic~2 is that, in this case,
we do not know of a similarly reliable classification.

Our strategy in proving Theorem~\ref{main} is to determine all pairs
of nilpotent Lie algebras $L_1,\ L_2$ with dimension at most~6,
such that the graded
algebras $\gr{L_1}$ and $\gr{L_2}$ associated
with the lower central series
are isomorphic. We know from~\cite{ru}
that this is a necessary condition for the isomorphism $\uea{L_1}\cong
\uea{L_2}$. Such pairs can be read off
from the list of nilpotent Lie algebras with dimension at most~6
in~\cite{degraaf}. Next, for all such pairs,
we either argue that $U(L_1)$ cannot be isomorphic to $U(L_2)$, or
we exhibit an explicit isomorphism between $U(L_1)$ and $U(L_2)$.
Initially, computer experiments played a role in determining the isomorphisms
between universal enveloping algebras of nilpotent Lie algebras. Recent work
by Eick~\cite{eick} describes a practical algorithm to decide isomorphism between
finite dimensional nilpotent associative algebras.
Her algorithm was implemented in the {\sf ModIsom} package~\cite{modisom} of
the {\sf GAP} computational algebra system~\cite{gap} and the implementation
works for algebras with  dimensions more than 100 over small finite fields.
We used this implementation to decide isomorphisms between the
finite-dimensional, nilpotent quotients $\aug{L}/\augp Lk$ of the
augmentation ideals $\aug L$ of $U(L)$ (see Section~\ref{prelim} for notation).
We remark here an interesting observation. If $L$ is nilpotent of class $c$ then based on our calculations the isomorphism type of the quotient $\aug{L}/\augp L{c+1}$ determines the isomorphism type of $L$. So, the question remains whether $\aug{L}/\augp L{c+1}$
determines the isomorphism type of $L$ in all dimensions.

\section{Preliminaries}\label{prelim}

In this section we summarize some important facts about universal enveloping
algebras of Lie algebras; see~\cite{dix} for a more detailed background.
We assume from now on that Lie algebras are finite-dimensional, even though
most of the results referred to in this section hold for a larger class of
Lie algebras.

Let $L$ be a Lie algebra over a field $\F$.
The universal enveloping algebra $\uea L$ of $L$ is
defined as follows. For $i\geq 0$, let $L^{\otimes i}$ denote the $i$-fold tensor
power of $L$ and set
$$
\mathcal T=\bigoplus_{i=0}^\infty L^{\otimes i}.
$$
The space $L^{\otimes 0}$ is one-dimensional, generated by the unit of $\mathcal T$,
and is usually identified with $\F$.
The sum $\mathcal T$ can be considered as an algebra over $\F$ with respect to
the obvious multiplication defined on the generators  of $\mathcal T$ as
$$
(v_1\otimes \cdots\otimes v_k)\cdot (u_1\otimes \cdots\otimes u_l)=
v_1\otimes \cdots\otimes v_k\otimes u_1\otimes \cdots\otimes u_l,
$$
for all $v_1,\ldots,v_k,u_1,\ldots,u_l\in L$. The algebra
$\mathcal T$ is usually referred to as the {\em tensor algebra} of $L$.
Let $\mathcal I$ denote the
two-sided ideal of $\mathcal T$ generated by elements of the form
$[u,v]-u\otimes v+v\otimes u$.  Then the {\em universal enveloping algebra}
$\uea L$ of $L$ is defined as the quotient $\mathcal T/\mathcal I$.
We view $L=L^{\otimes 1}$ as a  Lie subalgebra of $\uea L$.
Universal enveloping algebras have the following universal property.

\begin{lemma}\label{Uni}
Suppose that $L$ is a Lie algebra, $A$ is an associative algebra, and let
$\varphi:L\rightarrow A$ be a Lie algebra homomorphism. Then there is a unique
associative algebra homomorphism $\overline \varphi:\uea L\rightarrow A$ such
that $\overline\varphi|_L=\varphi$.
\end{lemma}

The linear subspace spanned by a subset $X$ of a vector space
is denoted by $\la X\ra_{\F}$.
Recall that the {\em center} $Z(L)$ of a Lie algebra $L$ is the ideal
$$
\la x\in L\ |\ [x,a]=0, \mbox{ for all }a\in L\ra_{\F}.
$$
The center of the universal enveloping algebra plays an important role in
our arguments. The {\em center} $Z(A)$ of an associative
algebra $A$ is defined as
$$
Z(A)=\la x\in A\ |\ ax=xa, \mbox{ for all } a\in A \ra_{\F}.
$$
It is clear that the subalgebra of $U(L)$ generated by the center
$Z(L)$ of $L$ lies in $Z(U(L))$.

Note that $\mathcal T_+=L^{\otimes 1}\oplus L^{\otimes 2}\oplus\cdots$ is a two-sided ideal in
$\mathcal T$, and the image of $\mathcal T_+$ in $\uea L$ is referred to as
the {\em augmentation ideal} of $\uea L$ and is denoted by $\aug L$.
The following is proved in~\cite[Lemma~2.1]{ru}.

\begin{lemma}\label{augisom}
For Lie algebras $L$ and $K$, $\uea L\cong \uea K$ if and only if
$\aug L\cong\aug K$.
\end{lemma}

Investigating the  isomorphism
between $\uea L$ and $\uea K$ will often be carried out, using
Lemma~\ref{augisom}, through studying
the isomorphism between $\aug K$ and $\aug L$.
For $i\geq 1$, let $\augp Li$ denote the ideal of $\aug L$ generated by the
products
of elements of $\aug L$ with length at least $i$. This way we obtain
a descending series in $\aug L$:
$$
\aug L\geq \augp L2\geq\cdots.
$$
The sequence $\augp Li$ is a filtration on $\aug L$: for $x\in \W^i(L)$
and $y\in\W^j(L)$ we have that $xy\in\augp L{i+j}$.

A basis for $\augp Li$
can usually be constructed as follows. Let $\lcs Li$ denote the $i$-th term of
the {\em lower central series} of $L$; that is $L^1=L$, and, for $i\geq 1$,
$L^{i+1}=[L^i,L]$. For an element $v\in L$, we define
the {\em weight} $w(v)$ of $v$ as the largest integer $i$
such that $v\in\lcs Li$. A basis $\mathcal B=\{v_1,\ldots,v_d\}$ of a Lie algebra
$L$ is said to be {\em homogeneous} if the basis elements
with weight at least $i$ form a basis for $\lcs Li$.
An element of $\uea L$ of the form $m=v_{i_1}v_{i_2}\ldots v_{i_k}$
with $i_1\leq i_2\leq\cdots\leq i_k$ is said
to be a {\em Poincar\'e-Birkhoff-Witt monomial},
or more briefly, a
{\em PBW monomial}, in $\mathcal B$.
 The {\em weight} $w(m)$ of
such a monomial is defined as $w(v_{i_1})+w(v_{i_2})+\cdots+w(v_{i_k})$.

\begin{theorem}[Proposition~3.1(1) in \cite{R}]\label{basis-w}
Let $L$ be a  Lie algebra with a homogeneous
basis $\mathcal B$ and let $t\geq 1$.
Then  the set of all PBW monomials in $\mathcal B$
with weight at least $t$
forms an $\F$-basis for $\augp Lt$, for every $t\geq 1$.
\end{theorem}

In this paper we are interested to discover,
for a pair of nilpotent non-isomorphic
Lie algebras  $L_1$ and $L_2$, if $U(L_1)$ can be isomorphic to $U(L_2)$.
Often we use the graded algebras $\gr{L_i}$ associated with the lower
central series of the $L_i$ to rule out the isomorphism between
$\uea{L_1}$ and $\uea{L_2}$. For a Lie algebra $L$, $\gr L$ is defined
as the algebra on the linear space
$$
\gr L=L/\lcs L2\oplus \lcs L2/\lcs L3\oplus\cdots
$$
with respect to the multiplication given by the rule
$$
[x+\lcs L{i+1},y+\lcs L{j+1}]=[x,y]+\lcs L{i+j+1} \quad
\mbox{for all $x\in \lcs Li$ and $y\in \lcs Lj$}.
$$
The following is proved in~\cite[Proposition~4.1]{ru}.

\begin{theorem}\label{gr}
For any Lie algebra $L$, the isomorphism type of $\uea L$  determines the isomorphism type of $\gr L$. Consequently,
if $L$ is nilpotent of class $2$, then the isomorphism type of $\uea L$ determines
the isomorphism type of  $L$.
\end{theorem}

Suppose that $L_1$ and $L_2$ are nilpotent Lie algebras such that
$U(L_1)\cong U(L_2)$. Then Theorem~\ref{gr} implies, for all $i$, that
$\dim \lcs {L_1}i=\dim \lcs {L_2}i$, and that the nilpotency classes of $L_1$ and $L_2$
coincide. The second statement of Theorem~\ref{gr}
is an easy consequence of
the first statement if $L$ is finite-dimensional.
Indeed, if $L$ is a finite-dimensional Lie algebra
of nilpotency class $2$, then it is always isomorphic to $\gr L$.
The same assertion without a restriction on the dimension is proved
in~\cite[Section~5]{ru}.

The following lemma can be found in~\cite[Lemma~5.1]{ru}.

\begin{lemma}\label{L^i}
Let $L$ and $K$ be Lie algebras and let $\varphi: \aug L\to \aug K$ be
an isomorphism.
Then $\varphi(L^i+\augp L{i+1})=K^i+\augp K{i+1}$, for every positive integer $i$.
\end{lemma}

\section{Nilpotent Lie algebras with dimension at most 5}\label{dim5}

In this section we determine all isomorphisms between the universal
enveloping algebras of nilpotent Lie algebras
with dimension at most 5.
A classification of such Lie algebras is well known and
can be found, for instance, in~\cite{degraaf}. The main result of this section
is the following.

\begin{theorem}\label{5main}
Let $L$ and $K$ be nilpotent Lie algebras of dimension at most~$5$ over a field $\F$ such that $\uea{L}\cong\uea{K}$.
Then one of the following must hold:
\begin{itemize}
\item[(i)] $L\cong K$;
\item[(ii)] $\char\,\F=2$; further $L$ and $K$ are isomorphic to the
Lie algebras $L_{5,3}$ and $L_{5,5}$ or they are isomorphic to the Lie algebras
$L_{5,6}$ and $L_{5,7}$ in Section~$5$ of~\cite{degraaf}.
\end{itemize}
\end{theorem}

Since there is a unique isomorphism class of nilpotent Lie  algebras
with dimension 1, and there is a unique such class with dimension 2,
the isomorphism problem of universal enveloping algebras is trivial in these
cases. Up to isomorphism,
there are two nilpotent Lie algebras with dimension $3$, an abelian,
and a non-abelian. By Theorem~\ref{gr} their universal enveloping algebras
must be non-isomorphic. The number of isomorphism classes of
4-dimensional nilpotent Lie algebras
is 3. One of these algebras is abelian, the second has nilpotency class 2, and
the third has nilpotency class 3. Again, by Theorem~\ref{gr},
their universal enveloping algebras are pairwise non-isomorphic.
This proves Theorem~\ref{5main} for Lie algebras of dimension at most~4.

There are 9 isomorphism classes of nilpotent Lie algebras with dimension 5 and
they are listed at the beginning of Section~4 in~\cite{degraaf}.
To simplify notation, we denote the Lie algebra $L_{5,i}$ in De Graaf's list by
$L_i$.
Inspecting this list, we find that the sequence  $(\dim\lcs L1,\dim\lcs L2,\ldots)$
for these Lie algebras, after omitting the tailing zeros,
are  as follows: $(5)$, $(5,1)$, $(5,2,1)$, $(5,1)$, $(5,2,1)$,
$(5,3,2,1)$, $(5,3,2,1)$, $(5,2)$, $(5,3,2)$. Therefore Theorem~\ref{gr}
implies that if $U(L_i)\cong U(L_j)$ then either $i=j$, or $\{i,j\}=\{3,5\}$,
or $\{i,j\}=\{6,7\}$. The Lie algebras that are involved in these
possible isomorphisms are as follows:
\begin{eqnarray*}
L_3&=&\left<x_1,x_2,x_3;x_4;x_5\ |\ [x_1,x_2]=x_4,\ [x_1,x_4]=x_5\right>;\\
L_5&=&\left< x_1,x_2,x_3;x_4;x_5\ |\ [x_1,x_2]=x_4,\ [x_1,x_4]=x_5,\
[x_2,x_3]=x_5\right>;\\
L_6&=&\left<x_1,x_2;x_3;x_4;x_5\ |\ [x_1,x_2]=x_3,\ [x_1,x_3]=x_4,\ [x_1,x_4]=x_5,\ [x_2,x_3]=x_5\right>;\\
L_7&=&\left<x_1,x_2;x_3;x_4;x_5\ |\ [x_1,x_2]=x_3,\ [x_1,x_3]=x_4,\ [x_1,x_4]=x_5\right>.
\end{eqnarray*}
Products that are zero are omitted from the multiplication tables above.
For instance $[x_1,x_3]=0$ in $L_3$ and in $L_5$.
Note that the bases of the Lie algebras given above
are homogeneous.  Further, basis elements of different weights are separated by a semicolon. The
multiplication tables given in~\cite{degraaf} are somewhat different from the
ones above, as we changed the orders of certain basis elements in order to
work with homogeneous bases.
A possible source of confusion is that
the symbols $x_i$ are used to denote elements of different Lie algebras,
but we believe that using different letters or introducing a subscript or
superscript would unnecessarily complicate the notation.

\begin{lemma}\label{5-dim}
If $\F$ is a filed of characteristic not 2, then  $\Omega(L_3)/\Omega^4(L_3)\not\cong
\Omega(L_5)/\Omega^4(L_5)$ and $\Omega(L_6)/\Omega^5(L_6)\not\cong
\Omega(L_7)/\Omega^5(L_7)$; consequently $\W(L_3)\not\cong\W(L_5)$
and $\W(L_6)\not\cong\W(L_7)$.
Otherwise, $\W(L_3)\cong \W(L_5)$ and $\W(L_6)\cong \W(L_7)$.
\end{lemma}
\begin{proof}
Suppose that $\char\,\F\neq 2$ and  set $B_i=\Omega(L_i)/\Omega^4(L_i)$,
for $i=3, 5$. First we show that $B_3\not\cong B_5$.
We claim that
$Z(B_5)\leq (B_5)^2$ while $Z(B_3)\not\leq (B_3)^2$, which
will imply that $B_3\not\cong B_5$.
As $x_3\in Z(B_3)\setminus(B_3)^2$, the second assertion of the claim
is valid.
In order to prove the first assertion, let $w\in Z(B_5)$.
We write $w$ as a linear combination of PBW monomials:
$$
w=\sum_{i_1\leq\cdots\leq i_n}\alpha_{i_1,\ldots,i_n}x_{i_1}\cdots x_{i_n}.
$$
Then $w\equiv\alpha_1 x_1+\alpha_2x_2+\alpha_3x_3\pmod{(B_5)^2}$.
Since $w$ is a central element in $B_5$, we have $[x_1, w]=[x_2, w]=0$,
and so
\begin{align*}
0=[x_1, w]&\equiv\alpha_2[x_1, x_2]= \alpha_2 x_4\pmod{(B_5)^3},\\
0=[x_2, w]&\equiv\alpha_1[x_2, x_1]=- \alpha_1 x_4\pmod{(B_5)^3}.
\end{align*}
Hence, $\alpha_1=\alpha_2=0$.
Since $x_5\in \lcs{L_5}3$, we have
\begin{align*}
0=[x_2, w]&=\alpha_3[x_2, x_3]+\alpha_{1,1} [x_2, x_1^2]+\alpha_{1,2}[x_2, x_1x_2]+\alpha_{1,3}[x_2, x_1x_3] \\
&=\alpha_3x_5+\alpha_{1,1}(-2x_1x_4+x_5)-\alpha_{1,2}x_2x_4-\alpha_{1,3}x_3x_4.
\end{align*}
Since $\char  \F\neq 2$, we deduce that $\alpha_3=0$. Thus, $w\in(B_5)^2$ as claimed.

Now set $B_i=\Omega(L_i)/\Omega^5(L_i)$, for $i=6, 7$.  Suppose, to the contrary, that
 $f:B_7 \rightarrow B_6$ is an isomorphism. Since $B_7$
is generated by $x_1$ and $x_2$, the map $f$ is determined by
the images $f(x_1)$ and $f(x_2)$.
As above, let us write $f(x_1)$ and $f(x_2)$ as linear combinations of
PBW monomials:
\begin{align*}
f(x_1)&=\sum_{i_1\leq\cdots\leq i_n}\alpha_{i_1,\ldots,i_n}x_{i_1}\cdots x_{i_n},\\
f(x_2)&=\sum_{i_1\leq\cdots\leq i_n}\beta_{i_1,\ldots,i_n}x_{i_1}\cdots x_{i_n}.
\end{align*}
As
\begin{align}\label{f[x1,x2]}
f([x_1,x_2])\equiv(\alpha_1\beta_2-\alpha_2\beta_1)[x_1,x_2]= (\alpha_1\beta_2-\alpha_2\beta_1)x_3\pmod{(B_6)^3},
\end{align}
we obtain that $\alpha_1\beta_2-\alpha_2\beta_1\neq 0$. Equation (\ref{f[x1,x2]}) also gives
$$
0=f([x_1,x_2,x_2])\equiv \beta_1(\alpha_1\beta_2-\alpha_2\beta_1)[x_1,x_2,x_1]=
-\beta_1(\alpha_1\beta_2-\alpha_2\beta_1)x_4 \pmod{(B_6)^4}.
$$
Since $\alpha_1\beta_2-\alpha_2\beta_1\neq 0$, we must have $\beta_1=0$.
This implies that $\alpha_1\neq 0$ and $\beta_2\neq 0$.
Furthermore, modulo $(B_6)^4$, we have
\begin{multline*}
f([x_1,x_2])\equiv
\alpha_1\beta_2x_3+(\alpha_1\beta_3+\alpha_2\beta_{1,1}-\alpha_{1,1}\beta_2)x_4+\\
(\alpha_1\beta_{1,2}-2\alpha_2\beta_{1,1}+2\alpha_{1,1}\beta_2)x_1x_3+
(2\alpha_1\beta_{2,2}-\alpha_2\beta_{1,2}+\alpha_{1,2}\beta_2)x_2x_3.
\end{multline*}
Hence $f([x_1,x_2,x_2])$ is equal to the following:
\begin{multline*}
(-\alpha_1\beta_2^2+\alpha_1\beta_2\beta_{1,1})x_5
-2\alpha_1\beta_2\beta_{1,1}x_1x_4-\alpha_1\beta_2\beta_{1,2} x_2x_4+\beta_2(\alpha_1\beta_{1,2}-2\alpha_2\beta_{1,1}+2\alpha_{1,1}\beta_2)x_3x_3.
\end{multline*}
As $[x_1,x_2,x_2]=0$ in $\W(L_7)$, we must have $f([x_1,x_2,x_2])=0$.
Since  $\char(\F)\neq 2$ and $\alpha_1\beta_2\neq 0$, we get $\beta_{1,1}=0$.
Thus the coefficient
of $x_5$ in $f([x_1,x_2,x_2])$ is $-\alpha_1\beta_2^2$. This, implies that $
f([x_1,x_2,x_2])\neq 0$, which is a contradiction. Thus $B_6\not\cong B_7$, as claimed.

Let us now assume that $\char\,\F=2$, and prove $\W(L_3)\cong\W(L_5)$ and $\W(L_6)\cong \W(L_7)$.
Note that the map from $L_3$ to $\W(L_5)$ induced by  $x_1\mapsto x_1$, $x_2\mapsto x_2$,
$x_3\mapsto x_3+x_1^2$ is a Lie homomorphism. Hence, by Lemma~\ref{Uni},
it can be extended to
a homomorphism $\varphi$ from $\W(L_3)$ to $\W(L_5)$.
Since $x_1$, $x_2$, and $x_3+x_1^2$ form
a generating set for $\W(L_5)$, $\varphi$ is onto. We need to show that
$\varphi$ is injective. Since
$\dim \W(L_3)/\W^i(L_3) = \dim \W(L_5)/\W^i(L_5)$ (see Theorem~\ref{basis-w}),
the homomorphism $\varphi$ induces an isomorphism
$$
\varphi_i:\W(L_3)/\W^i(L_3) \to \W(L_5)/\W^i(L_5),
$$
 for every $i\geq 1$.
Let  $x$ be a non-zero element in $\W(L_3)$ such that $\varphi(x)=0$. By Theorem \ref{basis-w}, there exists a positive integer $i$  such that $x\in \W^{i-1}(L_3)\setminus  \W^i(L_3)$.
As $\varphi_i$ is an isomorphism, we get $\varphi_i(x)\neq 0$,
which is a contradiction. Hence $\varphi$ is injective, and so $\W(L_3)\cong \W(L_5)$.

One can show precisely the same way that  $x_1\mapsto x_1$, $x_2\mapsto x_2+x_1^2$ extends to an isomorphism from
$\W(L_7)$ to $\W(L_6)$.
\end{proof}

Lemma~\ref{5-dim} in combination with Lemma~\ref{augisom},
and the argument preceding it proves Theorem~\ref{5main}.

\section{6-dimensional nilpotent Lie algebras}\label{dim6}

The isomorphisms between universal enveloping algebras of nilpotent Lie
algebras of dimension~6 in characteristic different from~2 are described by the following
theorem.

\begin{theorem}\label{6main}
Let  $L$ and $K$ be nilpotent Lie algebras of dimension~$6$ over a field $\F$ of characteristic not 2. If  $\uea{L}\cong\uea{K}$, then one of the following must hold.
\begin{itemize}
\item[(i)] $L\cong K$.
\item[(ii)] $\char\,\F=3$; further $L$ and $K$ are isomorphic to one
of the following pairs of Lie algebras in~\cite[Section~5]{degraaf}:
$L_{6,6}$ and $L_{6,11}$;
$L_{6,7}$ and $L_{6,12}$; $L_{6,17}$ and $L_{6,18}$; $L_{6,23}$ and $L_{6,25}$.
\end{itemize}
\end{theorem}

Throughout this section $\F$ denotes a field with characteristic different from 2.
The proof of Theorem~\ref{6main} is presented in this section.
De Graaf~\cite{degraaf} lists the isomorphism types of 6-dimensional
nilpotent Lie algebras over an arbitrary field $\F$ of characteristic different
from~2. De Graaf denotes these Lie algebras by $L_{6,i}$ or
$L_{6,i}(\varepsilon)$ where $1\leq i\leq 26$ and $\varepsilon$ is a field
element. To simplify notation and to distinguish
between the Lie algebras of Sections~\ref{dim5} and~\ref{dim6}
we will denote $L_{6,i}$ with $K_i$ and $L_{6,i}(\varepsilon)$ with
$K_i(\varepsilon)$.
Inspecting the list of Lie algebras, we obtain that the isomorphisms
among the graded Lie algebras associated with the lower central series of
these Lie algebras are as follows:

\begin{enumerate}
\item $\gr{K_{3}}\cong\gr{K_5}\cong \gr{K_{10}}$;\label{iso1}
\item $\gr{K_{6}}\cong\gr{K_7}\cong \gr{K_{11}}\cong\gr{K_{12}}\cong\gr{K_{13}}$; \label{iso2}
\item $\gr{K_{14}}\cong\gr{K_{16}}$;\label{iso3}
\item $\gr{K_{15}}\cong\gr{K_{17}}\cong \gr{K_{18}}$;\label{iso4}
\item $\gr{K_{23}}\cong\gr{K_{25}}$;\label{iso5}
\item $\gr{K_{24}(\varepsilon)}\cong\gr{K_{9}}, \mbox{ for all } \varepsilon\in\F$.\label{iso6}
\end{enumerate}
%

Let $L$ and $K$ be 6-dimensional  non-isomorphic nilpotent
Lie algebras over $\F$. If $\uea L\cong\uea K$ then, by Theorem~\ref{gr}, $L$ and $K$ must both occur in one of the lines~\eqref{iso1}--\eqref{iso6} above.
In Lemmas \ref{dim51}-\ref{lem-iso6} we examine the possible isomorphisms
between the universal enveloping algebras of the Lie algebras that
occur in one of the families above. Using Lemma \ref{augisom} we only examine
the possible isomorphisms between the augmentation ideals.

 As with the 5-dimensional Lie algebras, we change the multiplication tables of the
algebras presented in~\cite{degraaf} in order to work with homogeneous bases.
The bases elements of different weights are separated by  a semicolon. Further,
as in Section~\ref{dim5}, we omit products of the form
$[x_i,x_j]=0$  from the multiplication tables of the Lie algebras.
For $i=3,\ 5,\ 6,\ 7,\ 9,\ 10,\ 11,\ 12,\ 13,\ 14,\ 15,\ 16,\ 17,\
18,\ 23,\ 25$ we set $U_i=\uea{K_i}$ and $\Omega_i=\aug{K_i}$.
Further, let $U_{24}(\varepsilon)$
and $\Omega_{24}(\varepsilon)$ denote $\uea{K_{24}(\varepsilon)}$ and
$\aug{K_{24}(\varepsilon)}$, respectively.

\subsection{Family~\eqref{iso1}}

First we deal with the isomorphism $\gr{K_{3}}\cong\gr{K_5}\cong \gr{K_{10}}$
where
\begin{eqnarray*}
K_3&=&\left<x_1,x_2,x_3,x_4;x_5;x_6\ |\ [x_1,x_2]=x_5,\ [x_1,x_5]=x_6\right>;\\
K_5&=&\left<x_1,x_2,x_3,x_4;x_5;x_6\ |\ [x_1,x_2]=x_5,\ [x_1,x_5]=x_6,\
[x_2,x_3]=x_6\right>;\\
K_{10}&=&\left<x_1,x_2,x_3,x_4;x_5;x_6\ |\ [x_1,x_2]=x_5,\ [x_1,x_5]=x_6,\
[x_3,x_4]=x_6\right>.
\end{eqnarray*}

\begin{lemma}\label{dim51}
The algebras $\Omega_3$, $\Omega_5$, and $\Omega_{10}$ are pairwise non-isomorphic.
\end{lemma}
\begin{proof}
For $i=3,\ 5,\ 10$, let $B_i=\W_i/(\W_i)^4$.
It is enough to prove that the centers $Z_i=Z(B_i)$ have different dimensions.
The quotient $B_i$ is spanned by the images of the
PBW monomials with weight at most 3 and
we will identify such a monomial with
its image.
We claim that $Z_3=\left<x_3,x_4,(B_3)^3\right>_{\F}$. Set
$C=\left<x_3,x_4,(B_3)^3\right>_{\F}$.
Clearly,  $C\leq Z_3$. Let $z\in Z_3$. Then $z$ is a linear
combination of PBW monomials with weight at most~3.
As usual, $\alpha_{i_1,\ldots,i_n}$ is  the coefficient of $x_{i_1}\cdots x_{i_n}$ in the PBW representation of $z$. We may assume without
loss of generality that all monomials in $C$ occur with coefficient zero.
First we compute that
$$
[x_1,z]=\alpha_2 x_5+\alpha_{1,2} x_1x_5+2\alpha_{2,2}x_2x_5+\alpha_{2,3} x_3x_5+\alpha_{2,4} x_4x_5+\alpha_{5}x_6.
$$
We deduce that $\alpha_2=\alpha_{1,2}=\alpha_{2,2}=\alpha_{2,3}=\alpha_{2,4}=\alpha_{5}=0$.
Now
$$
[z,x_2]=\alpha_1x_5+2\alpha_{1,1}x_1x_5-\alpha_{1,1} x_6+\alpha_{1,3}x_3x_5+\alpha_{1,4}x_4x_5.
$$
This implies that  $\alpha_1=\alpha_{1,1}=\alpha_{1,3}=\alpha_{1.4}=0$.
Therefore $Z_3=C$ as claimed. Since $x_3$, $x_4$, $x_3x_3$, $x_3x_4$, $x_4x_4$
together with the PBW monomials with weight~3 form a basis for $C$, we
obtain that $\dim C=\dim Z_3=30$.
Similar argument shows that $\dim Z_5=29$, and $\dim Z_{10}=28$.
Thus the $Z_i$ have different dimensions, as required.
\end{proof}

\subsection{Family \eqref{iso2}}
We examine the following family of Lie algebras.

\begin{eqnarray*}
K_6&=&\left<x_1,x_2,x_3;x_4;x_5;x_6\ |\ [x_1,x_2]=x_4,\ [x_1,x_4]=x_5,\ [x_1,x_5]=x_6,\ [x_2,x_4]=x_6\right>;\\
K_7&=&\left<x_1,x_2,x_3;x_4;x_5;x_6\ |\  [x_1,x_2]=x_4,\ [x_1,x_4]=x_5,\ [x_1,x_5]=x_6\right>;\\
K_{11}&=&\left<x_1,x_2,x_3;x_4;x_5;x_6\ |\ [x_1,x_2]=x_4,\ [x_1,x_4]=x_5,\ [x_1,x_5]=x_6,\ [x_2,x_4]=x_6,\right.\\
&&\left.[x_2,x_3]=x_6\right>;\\
K_{12}&=&\left<x_1,x_2,x_3;x_4;x_5;x_6\ |\ [x_1,x_2]=x_4,\ [x_1,x_4]=x_5,\ [x_1,x_5]=x_6,\ [x_2,x_3]=x_6\ \right>;\\
K_{13}&=&\left<x_1,x_2,x_3;x_4;x_5;x_6\ |\ [x_1,x_2]=x_4,\ [x_1,x_4]=x_5,\ [x_1,x_5]=x_6,\ [x_2,x_3]=x_5,\right.\\
&&\left. [x_4,x_3]=x_6\right>.
\end{eqnarray*}

\begin{lemma}\label{fam2lem}
The enveloping algebras $\W_6$, $\W_7$, $\W_{11}$, $\W_{12}$, and
$\W_{13}$ are pairwise non-isomorphic provided that $\char\,\F\neq 3$.
If $\char\,\F=3$ then $\W_6\cong \W_{11}$ and $\W_7\cong \W_{12}$ and there
is no more isomorphism among the algebras in this family.
\end{lemma}
\begin{proof}
We claim, for $i\in\{6,\ 7,\ 11,\ 12\}$,  that $\W_{13}\not\cong \W_i$.
Note that $L_3\cong K_i/K_i^4$, while $L_5\cong K_{13}/K_{13}^4$,
where $L_3$ and $L_5$ are 5-dimensional Lie algebras defined in
Section~\ref{dim5}. Thus, $\W(L_3)/\W^5(L_3)\cong\Omega_i/(K_i^4+\Omega_i^5)$ and
$\W(L_5)/\W^5(L_5)\cong\Omega_{13}/(K_{13}^4+\Omega^5_{13})$.
Suppose that $f: \Omega_i\rightarrow \Omega_{13}$ is an isomorphism.
By Lemma~\ref{L^i}, $f(K_i^4+\Omega_i^5)=K_{13}^4+\Omega_{13}^5$. Hence, $f$
induces an isomorphism between $\Omega_{i}/(K_i^4+\Omega_i^5)$ and
 $\Omega_{13}/(K_{13}^4+\Omega_{13}^5)$.
Thus,
$$
\W(L_3)/\W^5(L_3)\cong \W(L_5)/\W^5(L_5).
$$
 However Lemma~\ref{5-dim} shows that $\W(L_3)/\W^4(L_3)\ncong \W(L_5)/\W^4(L_5)$. This contradiction implies that  $\Omega_{13}\not\cong\Omega_i$, as claimed.


Next we show that $\Omega_6\not\cong \Omega_7$ and that $\Omega_{11}\not\cong \Omega_{12}$.
First we argue that $\Omega_6\not\cong \Omega_7$. Suppose on the contrary that
$f:\Omega_7\rightarrow \Omega_6$ is an isomorphism. Then $f$ is
determined by the images $f(x_1)$, $f(x_2)$, $f(x_3)$.
Write $f(x_i)$ as a linear combination of PBW monomials and
let $\alpha_{i_1,\ldots,i_n}$, $\beta_{i_1,\ldots,i_n}$ and
$\gamma_{i_1,\ldots,i_n}$ denote the coefficients of $x_{i_1}\cdots x_{i_n}$
in $f(x_1)$, $f(x_2)$, $f(x_3)$, respectively.
Since $x_3\in Z(\Omega_7)$, we have
$$
0=[x_1,f(x_3)]\equiv\gamma_2x_4\bmod(\Omega_6)^3\quad\mbox{and}\quad
0=[x_2,f(x_3)]\equiv-\gamma_1x_4\bmod(\Omega_6)^3.
$$
Therefore $\gamma_1=\gamma_2=0$.
Let us compute modulo $(\Omega_6)^4$ that
$$
0=f([x_2,x_1,x_2])\equiv
\beta_1(\alpha_1\beta_2-\beta_1\alpha_2)x_5.
$$
If $\alpha_1\beta_2-\beta_1\alpha_2= 0$, then,
as $\gamma_1=\gamma_2=0$, it follows that $f(x_i)$ are linearly dependent modulo $(\Omega_6)^2$, which is
impossible. Thus $\beta_1=0$ and $\alpha_1\beta_2\neq 0$.
Now computation shows, modulo $(\Omega_6)^5$, that
\begin{multline*}
0=f([x_1,x_2,x_2])\equiv\\(-\alpha_1\beta_2^2+\alpha_1\beta_2\beta_{11})x_6-
2\alpha_1\beta_2\beta_{1,1}x_1x_5-\alpha_1\beta_2\beta_{1,2}x_2x_5-
\alpha_1\beta_2\beta_{1,3}x_3x_5+\delta\beta_2 x_4x_4
\end{multline*}
where $\delta=\alpha_1\beta_{1,2}-2\alpha_{2}\beta_{1,1}+2\alpha_{1,1}\beta_2$.
Considering the coefficients of $x_6$ and $x_1x_5$, we deduce that either $\alpha_1=0$ or $\beta_2=0$,
which is a contradiction.

Let us now show that $\Omega_{11}\not\cong \Omega_{12}$.
Assume by contradiction that $f:\Omega_{12}\rightarrow \Omega_{11}$ is
an isomorphism. Then $f$ is determined by the images
$f(x_1)$, $f(x_2)$, and $f(x_3)$.
As above, we write $f(x_i)$ as a linear combination of PBW monomials and
we let $\alpha_{i_1,\ldots,i_n}$, $\beta_{i_1,\ldots,i_n}$ and
$\gamma_{i_1,\ldots,i_n}$ denote the coefficients of $x_{i_1}\cdots x_{i_n}$
in $f(x_1)$, $f(x_2)$, $f(x_3)$, respectively.
Let us compute modulo
$(\Omega_{11})^3$ that
$$
0=f([x_1,x_3])\equiv(\alpha_1\gamma_2-\gamma_1\alpha_2)x_4
$$
and
$$
0\equiv f([x_2,x_3])\equiv(\beta_1\gamma_2-\beta_2\gamma_1)x_4.
$$
If the vector $(\gamma_1,\gamma_2)$ is not zero,
then the vectors $(\alpha_1,\alpha_2)$ and $(\beta_1,\beta_2)$
must be its scalar multiples and so the $f(x_i)$
are linearly dependent modulo $(\Omega_{11})^2$. Hence $\gamma_1=\gamma_2=0$.
Now we get a contradiction using the same argument as in the previous paragraph.

Now assume that $\char\,\F \neq 3$.
Set  $B_i=\Omega_i/(\Omega_i)^5$, for $i=6,\ 7,\ 11,\ 12$.
Let $Z_i$ denote the center of the $B_i$.
We claim that $\dim Z_6=\dim Z_7=29$ and $\dim Z_{11}=\dim Z_{12}=28$.
We only compute $\dim Z_7$, as the computation
of the other $Z_i$ is very similar.
We claim that
$$
Z_7=\left<x_3,x_3x_3,x_3x_3x_3,(B_7)^4\right>_{\F}.
$$
Set $C=\left<x_3,x_3x_3,x_3x_3x_3,(B_7)^4\right>_{\F}$. Clearly, $C\leq Z_7$.
Let $z\in Z_7$ and write $z$ as a linear combination of PBW monomials with
weight at most 4.
We may assume that all PBW monomials in $C$ occur with coefficient
$0$ in $z$. Let $\beta_{i_1,\ldots,i_k}$ be the coefficient of $x_{i_1}\cdots
x_{i_k}$. Then $\beta_3=\beta_{3,3}=\beta_{3,3,3}=0$.
First we obtain that
$$
0=[z,x_5]=\beta_1[x_1,x_5]=\beta_1x_6,
$$
and so $\beta_1=0$. Also, $[z,x_1^3]=-3\beta_2x_1^2x_4+3\beta_2x_1x_5-\beta_2x_6$. So, $\beta_2=0$.
Next compute that
\begin{multline*}
[z,x_1^2]=-2\beta_{1,2}x_1^2x_4-4\beta_{2,2}x_1x_2x_4-2\beta_{2,3}x_1x_3x_4
\\+
(\beta_{1,2}-2\beta_4)x_1x_5
+2\beta_{2,2}x_2x_5+\beta_{2,3}x_3x_5+2\beta_{2,2}x_4x_4+\beta_4x_6.
\end{multline*}
This
implies that $\beta_{1,2}=\beta_{2,2}=\beta_{2,3}=\beta_4=0$.
Further,
$$
[z,x_2^2]=4\beta_{1,1}x_1x_2x_4-2\beta_{1,1}x_2x_5-2\beta_{1,1}x_4x_4+
2\beta_{1,3}x_2x_3x_4.
$$
Thus $\beta_{1,1}=\beta_{1,3}=0$.
Then
\begin{multline*}
0=[z,x_1]=-\beta_{1,1,2}x_1x_1x_4-2\beta_{1,2,2}x_1x_2x_4-\beta_{1,2,3}
x_1x_3x_4-3\beta_{2,2,2}x_2x_2x_4-2\beta_{2,2,3}x_2x_3x_4\\
-\beta_{2,3,3}x_3x_3x_4
-\beta_{1,4}x_1x_5-\beta_{2,4}x_{4}^2-\beta_{2,4}x_2x_5-\beta_{3,4}x_3x_5-\beta_5x_6,
\end{multline*}
and so $\beta_{1,1,2}=\beta_{1,2,2}=\beta_{1,2,3}=\beta_{2,2,2}=
\beta_{2,2,3}=\beta_{2,3,3}=\beta_{1,4}=\beta_{2,4}=\beta_{3,4}=
\beta_5=0$.
Further
\begin{multline*}
0=[z,x_2]=3\beta_{1,1,1}x_1x_1x_4-3\beta_{1,1,1}x_1x_5+\beta_{1,1,1}x_6\\+2\beta_{1,1,3}x_1x_3x_4
-\beta_{1,1,3}x_3x_5+\beta_{1,3,3}x_3x_3x_4.
\end{multline*}
Hence $\beta_{1,1,1}=\beta_{1,1,3}=\beta_{1,3,3}=0$.
Thus $z\in C$, and so $C=Z_7$ as required. Similar
calculations show that
\begin{eqnarray*}
Z_6&=&\left<x_3,x_3x_3,x_3x_3x_3,(B_6)^4\right>_{\F};\\
Z_{11}&=&\left<x_3x_3,x_3x_3x_3,(B_6)^4\right>_{\F};\\
Z_{12}&=&\left<x_3x_3,x_3x_3x_3,(B_6)^4\right>_{\F}.
\end{eqnarray*}
Thus the dimensions of $Z_i$ are as claimed. Hence, if $\char\,\F \neq 3$ then,
 $\W_{6}\not\cong \W_{11},\ \W_{12}$ and  $\W_{7}\not\cong \W_{11},\ \W_{12}$.

Combining the results of the last two paragraph, we obtain that the algebras
$\W_6$, $\W_7$, $\W_{11}$, $\W_{12}$, $\W_{13}$
are pairwise non-isomorphic if $\char\,\F\neq 3$.

Now suppose that  $\char\,\F=3$.
We are required to show that $\W_{6}\not\cong \W_{12}$,
$\W_{7}\not\cong \W_{11}$, $\W_6\cong \W_{11}$ and that
$\W_7\cong \W_{12}$.
Suppose that $f:\Omega_{12}\rightarrow\Omega_6$ is an isomorphism. Write $f(x_2)$
and $f(x_4)$ as linear combinations of PBW monomials and
let
$\beta_{i_1,\ldots,i_n}$ and $\gamma_{i_1,\ldots,i_n}$ denote the coefficient
of $x_{i_1}\cdots x_{i_n}$ in $f(x_2)$ and $f(x_4)$ respectively.
By Lemma~\ref{L^i}, $f(x_4)\equiv\gamma_4 x_4\pmod{(\Omega_6)^3}$.
Then, modulo $(\Omega_6)^4$,
$0=f([x_2,x_4])\equiv\beta_1\gamma_4x_5$, which gives that $\beta_1=0$.
As $x_4=[x_1,x_2]$, we find that
$$
\gamma_{1,1,1}=\gamma_{1,1,2}=\gamma_{1,1,3}=
\gamma_{1,2,2}=\gamma_{1,2,3}=\gamma_{1,3,3}=\gamma_{2,2,2}=\gamma_{2,2,3}=
\gamma_{2,3,3}=\gamma_{3,3,3}=0.
$$
Then
$$
0=f([x_2,x_4])=(\beta_2\gamma_4-\beta_{1,1}\gamma_4)x_6-\beta_2\gamma_{1,4}x_4x_4+\beta_{1,2}\gamma_4x_2x_5+\beta_{1,3}\gamma_4x_3x_5-\beta_{1,1}\gamma_4 x_1x_5.
$$
This implies that $\beta_2\gamma_4=0$, and in turn that $\beta_2=0$.
However, this gives that $f(x_1),\ f(x_2),\ f(x_3)$ are linearly
dependent modulo $(\Omega_6)^2$, which is impossible. Hence
$\W_{6}\not\cong \W_{12}$, and very similar
argument shows that $\W_{7}\not\cong \W_{11}$.

Finally we need to show that if $\char\,\F=3$, then $\Omega_6\cong \Omega_{11}$ and $\Omega_7\cong \Omega_{12}$. The argument
presented in the proof of Lemma~\ref{5-dim}
shows that the map $x_1\mapsto x_1$,
$x_2\mapsto x_2$, and $x_3\mapsto x_3+x_1^3$ can be extended to
isomorphisms between $\Omega_6$ and $\Omega_{11}$ and between
$\Omega_7$ and $\Omega_{12}$.
\end{proof}

\subsection{Family~\eqref{iso3}}

\begin{eqnarray*}
K_{14}&=&\left<x_1,x_2;x_3;x_4;x_5;x_6\ |\ [x_1,x_2]=x_3,\ [x_1,x_3]=x_4,\
[x_1,x_4]=x_5, [x_2,x_5]=x_6\right.\\
&&\left.\ [x_3,x_4]=-x_6,\ [x_2,x_3]=x_5\right>;\\
K_{16}&=&\left<x_1,x_2;x_3;x_4;x_5;x_6\ |\ [x_1,x_2]=x_3,\ [x_1,x_3]=x_4,\
[x_1,x_4]=x_5,\ [x_2,x_5]=x_6,\right.\\
&&\left.[x_3,x_4]=-x_6\right>.
\end{eqnarray*}

\begin{lemma}\label{lem-iso3}
The algebras $\W_{14}$ and $\W_{16}$ are not isomorphic.
\end{lemma}
\begin{proof}
Note that $K_{14}$ and $K_{16}$ are algebras with maximal class. Further,
$K_{14}/(K_{14})^5 \cong L_6$ and $K_{16}/(K_{16})^5\cong L_7$.
Now the argument presented in the first paragraph of the proof
of Lemma~\ref{fam2lem} shows that $\W_{14}\not\cong \W_{16}$.
\end{proof}

\subsection{Family~\eqref{iso4}}
\begin{eqnarray*}
K_{15}&=&\left<x_1,x_2;x_3;x_4;x_5;x_6\ |\ [x_1,x_2]=x_3,\ [x_1,x_3]=x_4,\
[x_1,x_4]=x_5,\ [x_2,x_3]=x_5,\right.\\
&&\left.[x_1,x_5]=x_6,\ [x_2,x_4]=x_6\right>;\\
K_{17}&=&\left<x_1,x_2;x_3;x_4;x_5;x_6\ |\ [x_1,x_2]=x_3,\ [x_1,x_3]=x_4,\
[x_1,x_4]=x_5,\ [x_1,x_5]=x_6,\right.\\
&&\left.[x_2,x_3]=x_6\right>;\\
K_{18}&=&\left<x_1,x_2;x_3;x_4;x_5;x_6\ |\ [x_1,x_2]=x_3,\ [x_1,x_3]=x_4,\
[x_1,x_4]=x_5,\ [x_1,x_5]=x_6\right>.
\end{eqnarray*}

\begin{lemma}
If $\char\,\F\neq 3$ then $\W_{15}$, $\W_{17}$ and $\W_{18}$ are pairwise
non-isomorphic; otherwise the only isomorphism among
these algebras is $\W_{17}\cong \W_{18}$.
\end{lemma}
\begin{proof} Note that
$K_{15}/(K_{15})^5\cong L_6$ and $K_{17}/(K_{17})^5\cong L_7 \cong
K_{18}/(K_{18})^5$.
Applying the same method as in the proof of Lemma~\ref{fam2lem} yields that
$\W_{15}\not\cong \W_{17}$ and $\W_{15}\not\cong \W_{18}$.

Suppose that $\char\,\F\neq 3$, and let us show that $\W_{17}\not\cong \W_{18}$.
Assume, by contradiction, that $f: \W_{17}\to \W_{18}$ is an isomorphism.
Then $f$ is determined by the images $f(x_1)$ and $f(x_2)$.
Let $f(x_1)\equiv\a_{1}x_1 +\a_{2}x_2$ and $f(x_2)\equiv\b_{1}x_1 +\b _{2}x_2$  modulo $\W_{18}^2$.
Let $\delta=\a_1\b_2-\a_2\b_1$. Clearly, $\delta \neq 0$. Note that $f(x_3)\equiv\delta x_3 \pmod{\W_{18}^3}$.
Thus,  $f(x_4)\equiv\a_1\delta x_4 \pmod{\W_{18}^4}$ and  $f(x_6)\equiv\a_1^3 \delta x_6 \pmod{\W_{18}^6}$.
Also, $f([x_2, x_3])\equiv\b_1\delta x_4 \pmod{\W_{18}^4}$. We deduce that $\b_1=0$.
Now we write
\begin{align*}
f(x_2)& \equiv \b_2x_2+ u_0+x_1u_1+x_1^2u_2+x_1^3u_3 \pmod{\W_{18}^4},\\
f(x_3)&\equiv \delta x_3+ v_0+x_1v_1+x_1^2v_2+x_1^3v_3 \pmod{\W_{18}^4},
\end{align*}
where each $u_i$ and $v_i$ is a linear combination of (possibly trivial) PBW  monomials that do not involve $x_1$.
Since $f(x_3)\equiv \delta x_3 \pmod{\W_{18}^3}$, we deduce that weight of $v_0+x_1v_1+x_1^2v_2$ is at least 3. Similarly, weight of
$u_0+x_1u_1$ is at least 2. Note that $[u_0, v_i]=[u_i, v_0]=0$. So,
\begin{align}\label{brack-3-2}
0 \equiv f([x_3, x_2])&\equiv \b_2\sum_{i=1}^3 [x_1^i, x_2]v_i - \delta \sum_{j=1}^2[x_1^j, x_3]u_j \pmod{\W_{18}^5}.
\end{align}
Expanding out the commutators in Equation (\ref{brack-3-2}), we observe that $x_1^2x_3v_3$ is the unique term in Equation (\ref{brack-3-2}) that has  the highest exponent of $x_1$. Since $\char(\F)\neq 3$, this can happen only if $v_3\in \W_{18}$. Thus,
$x_1^2x_3v_3\in \W_{18}^5$.  The highest exponent of $x_1$ in Equation (\ref{brack-3-2}) then  appears in
$x_1x_3v_2$ and $x_1x_4u_2$. So, these terms have to cancel out with each other. We deduce that
$u_2\in \W_{18}$ and $v_2\in \W_{18}^2$.  Now Equation (\ref{brack-3-2}) reduces to
$\b_2 x_3v_1 \equiv\delta  x_4u_1 \pmod{\W_{18}^5}.$ This implies that $u_1\in \W_{18}^2$ and $v_1\in \W_{18}^3$.
So, if we  write
\begin{align*}
f(x_2)& \equiv\b_2x_2+ u_0+x_1u_1+x_1^2u_2+x_1^3u_3+x_1^4u_4 \pmod{\W_{18}^5},\\
f(x_3)& \equiv\delta x_3+ v_0+x_1v_1+x_1^2v_2+x_1^3v_3+x_1^4v_4 \pmod{\W_{18}^5},
\end{align*}
then each $u_i$ and $v_i$ is a linear combination of (possibly trivial) PBW  monomials that do not involve $x_1$,
$u_0$ has weight at least 2, $v_0$ has weight at least 3,  weight of $x_1v_1+x_1^2v_2+x_1^3v_3+x_1^4v_4$ is at least 4, and  weight of
$x_1u_1+x_1^2u_2$ is at least 3. Thus,
\begin{align}\label{brack-3-3}
-\a_1^3 \delta x_6=f([x_3, x_2])& \equiv\b_2\sum_{i=1}^4 [x_1^i, x_2]v_i - \delta \sum_{j=1}^3[x_1^j, x_3]u_j \pmod{\W_{18}^6}.
\end{align}
Arguing as in Equation (\ref{brack-3-2}), we deduce that Equation (\ref{brack-3-3}) reduces to the following:
\begin{align*}
-\a_1^3 \delta x_6& \equiv\b_2 x_3 v_1 - \delta x_4u_j  \pmod{\W_{18}^6}.
\end{align*}
The later is possible only if $\a_1=0$ or $\delta=0$. Since we have already proved that $\b_1=0$ it follows that
$\delta=0$ which is a contradiction.

The argument in the proof of Lemma~\ref{5-dim} shows, for $\char(\F)=3$, that
the map  $x_1\mapsto x_1$, $x_2\mapsto x_2+x_1^3$ can be extended to
 an isomorphism between $\W_{17}$ and $\W_{18}$.
\end{proof}

\subsection{Family~\eqref{iso5}}
\begin{eqnarray*}
K_{23}&=&\left<x_1,x_2,x_3;x_4,x_5;x_6\ |\ [x_1,x_2]=x_4,\ [x_1,x_4]=x_6,\
[x_1,x_3]=x_5,\ [x_2,x_3]=x_6\right>;\\
K_{25}&=&\left<x_1,x_2,x_3;x_4,x_5;x_6\ |\ [x_1,x_2]=x_4,\ [x_1,x_4]=x_6,\
[x_1, x_3]=x_5\right>.
\end{eqnarray*}

\begin{lemma}
The algebras  $\W_{23}$ and $\W_{25}$ are not isomorphic.
\end{lemma}
\begin{proof}
Suppose that $\char\,\F\neq 3$ and assume, by contradiction, that  $f:\W_{25} \to \W_{23}$ is an isomorphism.
For $i=1,\ 2,\ 3$, write $f(x_i)$ as a linear combination of PBW
monomials and assume that $\alpha_{i_1,\ldots,i_n}$, $\beta_{i_1,\ldots,i_n}$,
and $\gamma_{i_1,\ldots,i_n}$ are the coefficients of $x_{i_1}\cdots x_{i_n}$
in $f(x_1)$, $f(x_2)$, and $f(x_3)$, respectively.
Then
$$
0=f([x_2, x_3])\equiv(\beta_1\gamma_2-\beta_2\gamma_1)x_4 +
(\beta_1\gamma_3-\beta_3\gamma_1)x_5 \pmod{\W_{23}^3},
$$
and hence $\beta_1\gamma_2-\beta_2\gamma_3=\beta_1\gamma_3-\beta_3\gamma_1=0$.
Since $f$ induces an isomorphism between $\Omega_{25}/(\Omega_{25})^2$
and
$\Omega_{23}/(\Omega_{23})^2$, we obtain that $\beta_1=\gamma_1=0$.
Note that
$$
f(x_5)=f([x_1, x_3])\equiv\alpha_1\gamma_2x_4+\a_1\gamma_3x_5 \pmod{\W_{23}^3}.
$$
Thus,
$0=f([x_1, x_5])\equiv\a_1^2\gamma_2x_6\pmod{(\W_{23})^4}$, which gives that
$\gamma_2=0$.
Now we calculate $f([x_2, x_3])$ modulo $(\W_{23})^4$
to show that $\beta_2\gamma_3=0$.
As $[x_4, \W_{23}]\leq (\W_{23})^4$, we have, modulo $(\W_{23})^4$,  that
\begin{multline*}
f([x_2, x_3])\equiv
\beta_2\gamma_3x_6+ \beta_2\gamma_{1,1}(x_6-2x_1x_4) -\beta_2\gamma_{1,2}
x_2x_4  -\beta_2\gamma_{1,3} x_3x_4\\
+2(\beta_{1,1}\gamma_3 -\beta_3\gamma_{1,1}) x_1x_5 + (\beta_{1,2}\gamma_3-\beta_3\gamma_{1,2})x_2x_5 +(\beta_{1,3}\gamma_3 -\beta_3\gamma_{1,3}) x_3 x_5.
\end{multline*}
This gives that $\beta_2\gamma_3=0$ which implies that the images
$f(x_2)$, $f(x_2)$, $f(x_2)$
are linearly dependent modulo $(\Omega_{23})^2$, which is a contradiction.
\end{proof}

\subsection{Family~\eqref{iso6}}
We consider the following Lie algebras:
\begin{eqnarray*}
K_{9}&=&\left<x_1,x_2,x_3;x_4;x_5,x_6\ |\ [x_1,x_2]=x_4,\ [x_1,x_4]=x_5,\ [x_2,x_4]=x_6\right>;\\K_{24}(\varepsilon)
&=&\left<x_1,x_2,x_3;x_4;x_5,x_6\ |\ [x_1,x_2]=x_4,\ [x_1,x_4]=x_5,\
[x_1,x_3]=\varepsilon x_6,\ [x_2,x_4]=x_6,\right.\\
&&\left.[x_2,x_3]=x_5\right>.\\
\end{eqnarray*}
The family $K_{24}(\varepsilon)$ is a parametric family of Lie
algebras such that $K_{24}(\varepsilon_1)\cong K_{24}(\varepsilon_2)$ if
and only if there is a $\nu\in\F$ such that
$\varepsilon_1\nu^2=\varepsilon_2$ (see~\cite{degraaf}).

\begin{lemma}\label{lem-iso6}
The algebras   $\W_9$ and $\W_{24}(\varepsilon)$ are not isomorphic, for all  $\varepsilon\in\F$.  Further,  $\W_{24}(\varepsilon_1)\cong \W_{24}(\varepsilon_2)$ if and only if $K_{24}(\varepsilon_1)\cong K_{24}(\varepsilon_2)$,
for every $\varepsilon_1,\ \varepsilon_2
\in\F$.
\end{lemma}
\begin{proof}
Since $x_3\in Z(\Omega_9)\setminus(\Omega_9)^2$, we have that
$Z(\Omega_9)\not\leq (\Omega_9)^2$. Let $\varepsilon\in\F$.
We claim that $Z(\Omega_{24}(\varepsilon))\leq (\Omega_{24}(\varepsilon))^2$.
Let $z\in Z(\Omega_{24}(\varepsilon))$ and write $z$ as a linear combination
of PBW monomials in which $\alpha_{i_1,\ldots,i_n}$ denotes the coefficient
of $x_1\cdots x_n$.
First we compute, modulo $(\Omega_{24}(\varepsilon))^3$, that
$0=[z, x_3]\equiv\alpha_1\varepsilon x_6+ \alpha_2 x_5$, which gives that
$\alpha_2=0$.
Then,
$$
[z,x_2]\equiv \alpha_1 x_4+(-\alpha_3-\alpha_{1,1})x_5+2\alpha_{1,1}x_1x_4+\alpha_{1,2}x_2x_4+\alpha_{1,3}x_3x_4 \pmod{\Omega^4_{24}(\varepsilon)},
$$
which shows that $\alpha_1=\alpha_3=0$. Hence $z\in (\Omega_{24}(\varepsilon))^2$,
and so $Z(\Omega_{24}(\varepsilon))\leq (\Omega_{24}(\varepsilon))^2$,
as claimed. This implies that $\Omega_9\not\cong \Omega_{24}(\varepsilon)$.

Let us now prove the second assertion of the lemma.
Without loss of generality, we assume that $\e_2\neq 0$.
Let $K=K_{24}(\e_2)$ and $\W=\W(K)$.
Suppose that $f: \Omega_{24}(\varepsilon_1)\to \W$ is an algebra isomorphism.
As usual, for $i=1,\ 2,\ 3$, we write the images $f(x_i)$ as linear
combinations of PBW monomials, and let $\alpha_{i_1,\ldots,i_n}$,
$\beta_{i_1,\ldots,i_n}$ and $\gamma_{i_1,\ldots,i_n}$ denote the
coefficients of $x_{i_1}\cdots x_{i_n}$ in $f(x_1)$, $f(x_2)$, and
$f(x_3)$, respectively.
Since $f([x_1, x_3]),\ f([x_2,x_3])\in \W^3$, we deduce that $\a_1\gamma_2-\a_2\gamma_1=0$ and that
$\beta_1\gamma_2-\beta_2\gamma_1=0$. Since the images
$f(x_1)$, $f(x_2)$, $f(x_2)$ are linearly independent
modulo $\Omega^2$, this gives that $\gamma_1=\gamma_2=0$.
Set $\delta=\a_1\beta_2-\a_2\beta_1$. Since $f(x_1)$, $f(x_2)$,
and $f(x_3)$ are linearly independent modulo $\Omega^2$,
we have $\delta\neq 0$. Further,
$$
f(x_4)=[f(x_1), f(x_2)]\equiv\delta[x_1, x_2]=\delta x_4 \pmod{\W^3}.
$$
Thus,  $f([x_1,x_4])\equiv\alpha_1\delta x_5+\alpha_2\delta x_6 \pmod{\W^4}$.
So, modulo $\W^4$, we have,
\begin{multline*}
f([x_2,x_3])\equiv(\beta_1\gamma_4+\beta_2\gamma_3+\beta_2\gamma_{1,1})x_5+
(\varepsilon_2\beta_1\gamma_3-\beta_1\gamma_{2,2}+\beta_2\gamma_{1,2}+\beta_2\gamma_4)x_6+\\
(\beta_1\gamma_{1,2}-2\beta_2\gamma_{1,1})x_1x_4+
(2\beta_1\gamma_{2,2}-\beta_2\gamma_{1,2})x_2x_4+
(\beta_1\gamma_{2,3}-\beta_2\gamma_{1,3})x_3x_4.
\end{multline*}
As $f([x_1,x_4])=f([x_2,x_3])$ we obtain that following equations:
\begin{eqnarray}
\label{11}\beta_1\gamma_4+\beta_2\gamma_3+\beta_2\gamma_{1,1}&=&\alpha_1\delta;\\
\label{12}\varepsilon_2\beta_1\gamma_3-\beta_1\gamma_{2,2}+\beta_2\gamma_{1,2}+\beta_2\gamma_4 &=& \alpha_2\delta;\\
\label{13}\beta_1\gamma_{1,2}-2\beta_2\gamma_{1,1}&=&0;\\
\label{14}2\beta_1\gamma_{2,2}-\beta_2\gamma_{1,2}&=&0;\\
\label{15}\beta_1\gamma_{2,3}-\beta_2\gamma_{1,3}&=&0.
\end{eqnarray}
Now we use the relation $\varepsilon_1[x_2,x_4]=[x_1,x_3]$. So, modulo $\W^4$, we have
\begin{align*}
f([x_2,x_4])\equiv& \beta_1\delta x_5+\beta_2\delta x_6,\\
f([x_1,x_3])\equiv& (\alpha_1\gamma_4+\alpha_2\gamma_3+\alpha_2\gamma_{1,1})x_5+
(\varepsilon_2\alpha_1\gamma_3-\alpha_1\gamma_{2,2}+\alpha_2\gamma_{1,2}+\alpha_2\gamma_4)x_6+\\
&(\alpha_1\gamma_{1,2}-2\alpha_2\gamma_{1,1})x_1x_4+
(2\alpha_1\gamma_{2,2}-\alpha_2\gamma_{1,2})x_2x_4+
(\alpha_1\gamma_{2,3}-\alpha_2\gamma_{1,3})x_3x_4.
\end{align*}
We get the following equations:
\begin{eqnarray}
\label{21}\alpha_1\gamma_4+\alpha_2\gamma_3+\alpha_2\gamma_{1,1} &=&
\varepsilon_1\beta_1\delta;\\
\label{22}\varepsilon_2\alpha_1\gamma_3-\alpha_1\gamma_{2,2}+\alpha_2\gamma_{1,2}+\alpha_2\gamma_4 &=& \varepsilon_1\beta_2\delta;\\
\label{23}\alpha_1\gamma_{1,2}-2\alpha_2\gamma_{1,1}&=&0;\\
\label{24}2\alpha_1\gamma_{2,2}-\alpha_2\gamma_{1,2}  & =& 0;\\
\label{25}\alpha_1\gamma_{2,3}-\alpha_2\gamma_{1,3}&=&0.
\end{eqnarray}
Equations~\eqref{13} and~\eqref{23} imply that
$\gamma_{1,1}=\gamma_{1,2}=0$. Similarly $\gamma_{2,2}=
\gamma_{2,3}=\gamma_{1,3}=0$. Thus the system of
equations above are reduced to the
following:
\begin{eqnarray}
\label{31}\beta_1\gamma_4+\beta_2\gamma_3&=&\alpha_1\delta;\\
\label{32}\varepsilon_2\beta_1\gamma_3+\beta_2\gamma_4 &=& \alpha_2\delta;\\
\label{33}\alpha_1\gamma_4+\alpha_2\gamma_3 &=&
\varepsilon_1\beta_1\delta;\\
\label{34}\varepsilon_2\alpha_1\gamma_3+\alpha_2\gamma_4 &=& \varepsilon_1\beta_2\delta.
\end{eqnarray}
Set
\begin{eqnarray*}
p_1&=&(-1/2)\varepsilon_1\beta_2\delta^{-2}-
(1/2)\varepsilon_2\alpha_1\gamma_3\delta^{-3};\\
p_2&=&(1/2)\alpha_1\gamma_4\delta^{-3} + \alpha_2\gamma_3\delta^{-3};\\
p_3&=&\alpha_2\delta^{-2} - (1/2)\beta_2\gamma_4\delta^{-3}\\
p_4&=&(-1/2)\alpha_1\delta^{-2} - (1/2)\beta_2\gamma_3\delta^{-3}.
\end{eqnarray*}
We can check that
\begin{multline*}
p_1(\beta_1\gamma_4+\beta_2\gamma_3-\alpha_1\delta)+p_2(\varepsilon_2\beta_1\gamma_3+\beta_2\gamma_4-\alpha_2\delta)+\\
p_3(\alpha_1\gamma_4+\alpha_2\gamma_3 -\varepsilon_1\beta_1\delta)+p_4(\varepsilon_2\alpha_1\gamma_3+\alpha_2\gamma_4 - \varepsilon_1\beta_2\delta)=\varepsilon_1 - \varepsilon_2\gamma_3^2\delta^{-2}.
\end{multline*}
Thus, considering that $\delta\neq 0$,
the equation $\varepsilon_1 = \varepsilon_2\gamma_3^2\delta^{-2}$ follows
from the equations~\eqref{31}--\eqref{34}. However this implies that
$K_{24}(\varepsilon_1)\cong K_{24}(\varepsilon_2)$ as claimed.
\end{proof}

\section*{Acknowledgments}
The second author would like to thank the  Centro de \'Algebra da Universidade de Lisboa for hosting his visit in April 2010 while this work was completed.


\begin{thebibliography}{XXX}

\bibitem[CKL]{ckl}
J. Chun, T. Kajiwara, J. Lee.
\newblock
Isomorphism Theorem on Low Dimensional Lie Algebras.
\newblock
{\em Pacific J. Math.} {\bf 214} (2004), no. 1, 17--21.

\bibitem[D]{dix}
J. Dixmier.
\newblock
{\em Enveloping algebras.}
\newblock
Revised reprint of the 1977 translation. Graduate Studies in Mathematics, 11. American Mathematical Society, Providence, RI, 1996. xx+379 pp.

\bibitem[E]{eick}
B. Eick.
\newblock
Computing automorphism groups and testing isomorphisms for modular group algebras.
\newblock{\em J. Algebra}
{\bf 320} (2008), no. 11, 3895--3910.

\bibitem[MI]{modisom}
B. Eick.
\newblock
{\sf ModIsom} --- a {\sf  GAP} package, version~1.1, May 2009, www-public.tu-bs.de:8080/~beick/soft/modisom.

\bibitem[GAP]{gap}
\newblock
The GAP Group, GAP -- Groups, Algorithms, and Programming, Version 4.4.12; 2008. (http://www.gap-system.org)

\bibitem[dG]{degraaf} W.A. de Graaf.
\newblock
Classification of 6-dimensional nilpotent Lie algebras over fields of characteristic not 2.
\newblock {\em J. Algebra} {\bf 309} (2007), no. 2, 640--653.

\bibitem[KL]{KL}
G.R. Krause and T.H. Lenagan.
\newblock
{\em Growth of algebras and Gelfand-Kirillov dimension.}
\newblock
Revised edition. Graduate Studies in Mathematics, 22. American Mathematical Society, Providence, RI, 2000. x+212 pp.


\bibitem[M]{mal}
P. Malcolmson.
\newblock
Enveloping Algebras of Simple Three-Dimensional Lie Algebras.
\newblock
{\em J. Algebra} {\bf 146} (1992), 210--218.

\bibitem[R]{R}
D. Riley.
\newblock The dimension subalgebra problem for enveloping algebras of Lie superalgebras.
\newblock {\em Proc. Amer. Math. Soc.} {\bf 123} (1995), no. 10, 2975--2980.

\bibitem[RU]{ru}
D. Riley and H. Usefi.
\newblock The isomorphism problem for universal enveloping algebras of {L}ie   algebras.
\newblock {\em Algebr. Represent. Theory}, 10(6):517--532, 2007.

\bibitem[Sch]{sch}
C. Schneider.
\newblock
A computer-based approach to the classification of nilpotent Lie algebras.
\newblock
{\em Experiment. Math.} {\bf 14} (2005), no. 2, 153--160.

\end{thebibliography}

\end{document}